\documentclass[11pt]{article}
\usepackage{color}\usepackage{graphicx}
\usepackage{amsfonts} \usepackage{texdraw} \usepackage{amssymb,
amsmath, amsthm} \usepackage{epsfig} \usepackage[english]{babel}
\usepackage{latexsym} \usepackage{amscd}

\newtheorem{theorem}{Theorem}
 
\newtheorem{lemma}{Lemma} \newtheorem{proposition}{Proposition}
 
\newtheorem{remark}{Remark} 
\def\neweq#1{\begin{equation}\label{#1}} \def\endeq{\end{equation}}
\def\eq#1{(\ref{#1})} 
\def\endproof{\hfill $\Box$\par\vskip3mm} 
\newcommand\B{{\bf B}} 
\newcommand{\R}{\mathbb{R}}

\def\sr{\xi(x)}

\begin{document}

\title{On a nonlinear nonlocal hyperbolic system\\
modeling suspension bridges}

\author{Gianni ARIOLI $^\sharp$ - Filippo GAZZOLA $^\dagger$}
\date{}
\maketitle

\begin{center}
{\small $^\sharp$ MOX-- Modellistica e Calcolo Scientifico\\
        $^{\sharp,\dagger}$ Dipartimento di Matematica -- Politecnico di Milano \\
        Piazza Leonardo da Vinci 32 - 20133 Milano, Italy\\
{\tt gianni.arioli@polimi.it\ -\ filippo.gazzola@polimi.it}}
\end{center}

\begin{abstract}
We suggest a new model for the dynamics of a suspension bridge through a system of nonlinear nonlocal hyperbolic differential equations.
The equations are of
second and fourth order in space and describe the behavior of the main components of the bridge: the deck, the sustaining cables and the connecting
hangers. We perform a careful energy balance and we derive the equations from a variational principle. We then prove existence and uniqueness for
the resulting problem.
\end{abstract}

\section{Introduction}

The main span of a suspension bridge is a complex structure composed by several interacting components, see Figure \ref{suspension}.
Four towers sustain two cables that, in turn, sustain the hangers. At their lower endpoint the hangers are linked to the
roadway and sustain it from above. The roadway (or deck) may be seen as a thin rectangular plate.
The hangers are hooked to the cables and the roadway is hooked to the hangers; the weight of the deck deforms the cable and stretches the hangers that
exert a restoring action on the roadway.

\begin{figure}[ht]
\begin{center}
{\includegraphics[height=35mm, width=125mm]{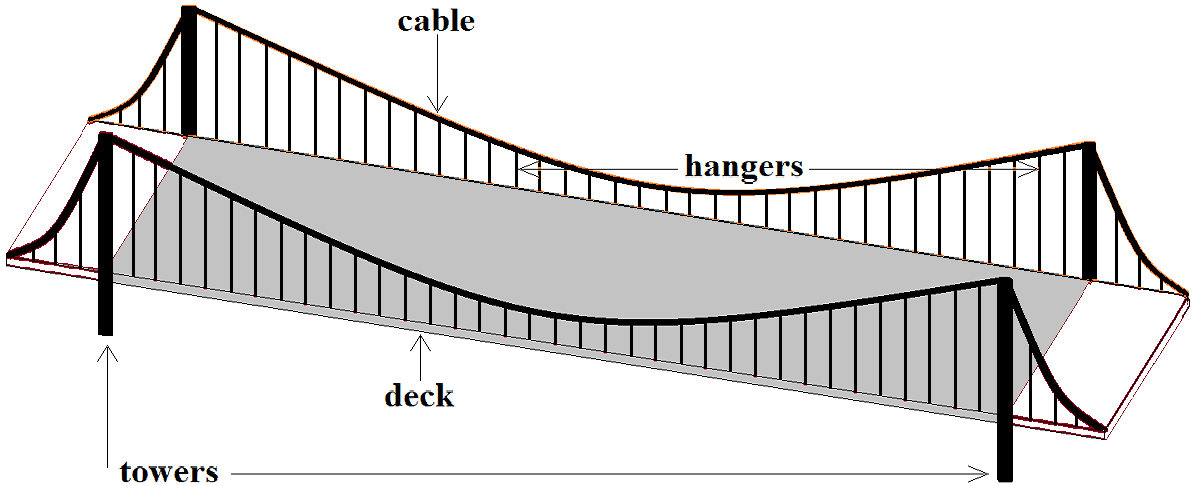}}
\caption{Sketch of a suspension bridge.}\label{suspension}
\end{center}
\end{figure}

This complex structure is extremely interesting from a mathematical point of view. It appears quite challenging to find a model
simple enough to be mathematically tractable and sufficiently accurate to display the main features of a real bridge. One of the main problems of
suspension bridges is their instability, in particular they are prone to torsional oscillations.
This is apparent in videos available on the web \cite{tacoma} as well as in many other events,
see e.g.\ \cite{bookgaz} for a survey. Partial explanations of the instability are based on aerodynamic effects such as vortex shedding, flutter,
parametric resonance, aerodynamic forces, see \cite{billah}. But none of these theories explains how a longitudinal oscillation can rapidly switch
into a torsional one. Scanlan \cite[p.209]{scanlan2} writes that the switch is due to {\em some fortuitous condition}, a justification which does not seem
very scientific. And, indeed, McKenna \cite{mckdcds} raises many doubts about the mathematical models and motivations used in classical literature.
In a recent paper \cite{argaz} we suggested that the onset of the torsional instability could be of purely structural nature. The model adopted there was
fairly simplified and did not describe accurately the nonlinear behavior of the cables+hangers restoring force. By using Poincar\'e-type maps
we were able to show that a longitudinal oscillation in an ideally isolated bridge may switch rapidly into a torsional oscillation when enough energy
is present within the structure.\par
Aiming to display the same phenomenon of transfer of energy between longitudinal and torsional oscillations in an actual bridge, in this paper we suggest
a new and realistic model which takes into account all the main components of the bridge and its quantitative structural parameters.
We perform an energy balance and we derive the corresponding
Euler-Lagrange equations by variational principles. The dynamics of the bridge is modeled through
a system of four hyperbolic equations, three of them of second order in space and one of fourth order. This system also contains nonlinear couplings
(the restoring action of the hangers) and nonlocal terms (the elongation of the extensible cables). In fact, the full energy balance is rather complicated,
therefore we approximate the equations with a first order expansion and we prove that the initial-boundary value problem for the hyperbolic system admits a unique
solution. In a forthcoming paper \cite{futuro} we will study higher order approximations.\par
In Section \ref{physmodel} we derive an accurate model and we proceed in two steps. First, in Sections \ref{beammodel} and \ref{Lagrangebeam}
we compute the Lagrangian of a system composed by a one-dimensional beam suspended to a cable through hangers,
see Figure \ref{cablebeam}. In the classical literature
\cite{biot,irvine,melan} one neglects either the mass of the cable, in which case the load is distributed per horizontal unit and the cable takes
the shape of a parabola, or the mass of the beam, in which case the load is distributed per unit length and the cable takes the shape of a catenary.
Since none of these masses is negligible, we take here into account both their contributions. This gives rise to a somehow intermediate shape and
the linearized problem becomes a Sturm-Liouville problem with a
weight. Second, in Section \ref{energybalance} we extend the result to the full suspension bridge
by modelling the roadway as a degenerate plate, that is, a beam representing the midline of the roadway
with cross sections which are free to rotate around the beam; simplified equations representing this model were
previously analyzed in \cite{berchiog,moore}. The midline corresponds to the barycenter of the cross sections.
The cross sections are seen as rods that can rotate around
their barycenter. The endpoints of the cross sections are the edges of the plate that are connected to the sustaining cables through the hangers.\par
Then we derive the Euler-Lagrange equations by variational methods, that is, by finding critical points
of the Lagrangian. The resulting equations form a system of semilinear hyperbolic equations, three of them being of second order and one of fourth order,
see \eq{primordine}. All the coefficients are continuous time-independent functions and \eq{primordine} is a nonlocal differential problem due to the
presence of the integral term representing the cables elongations. A nonlinear term is responsible for the coupling between the equations. These
peculiarities show that \eq{primordine} is not a standard problem and therefore we need to prove existence and uniqueness of a solution.
This is done in Section \ref{wellpos}.

\section{The physical model}\label{physmodel}

Throughout this paper we denote the partial derivatives of a function $y=y(x,t)$ by
$$y'=\frac{\partial y}{\partial x}\ ,\qquad \dot{y}=\frac{\partial y}{\partial t}$$
and similarly for higher order derivatives.

\subsection{A beam sustained by a cable through hangers}\label{beammodel}

As a first step, we derive the Lagrangian of a suspended beam connected to a sustaining cable through hangers. Let $w$ be the position of the beam and
$p-s$ be the position of the cable, as in Figure \ref{cablebeam}. The purpose of this section is to derive the equation of this cable-hangers-beam system.
\begin{figure}[ht]
\begin{center}
{\includegraphics[height=36mm, width=148mm]{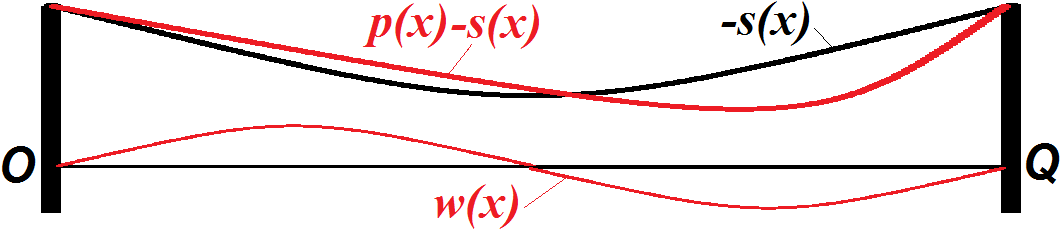}}
\caption{Cable-beam structure modeling a suspension bridge.}\label{cablebeam}
\end{center}
\end{figure}

We model the cable as a perfectly flexible string subject to vertical loads. If we assume that the string has no resistance to bending, the only internal
force is the tension $F=F(x)$ of the string which acts tangentially to the position of the curve representing the string. In Figure \ref{rope}
\begin{figure}[ht]
\begin{center}
{\includegraphics[height=48mm, width=100mm]{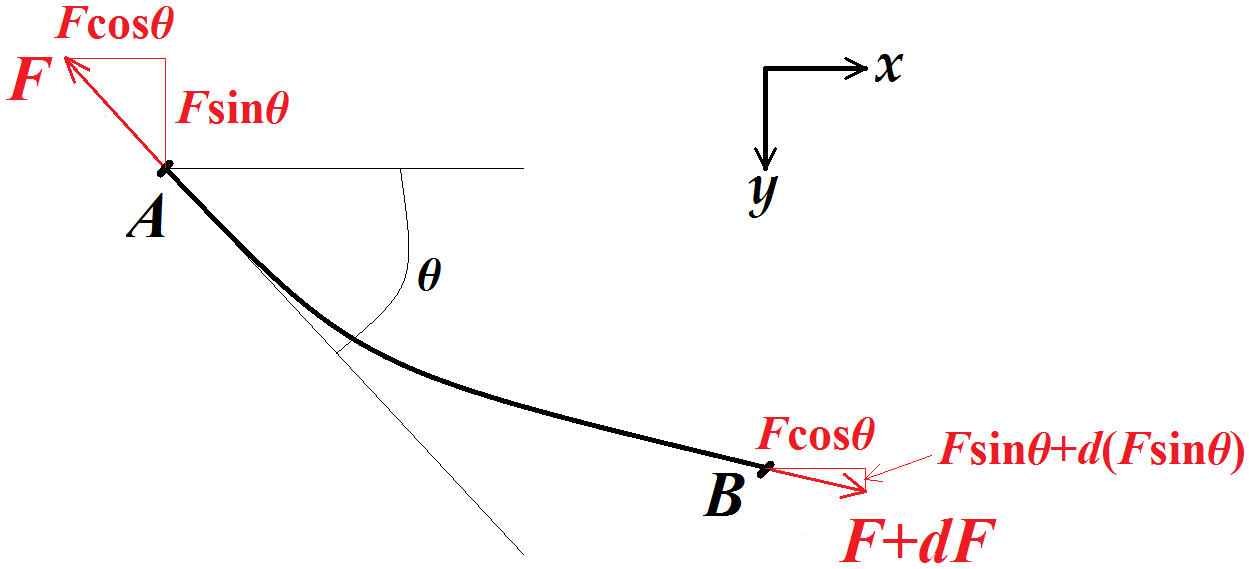}}
\caption{Equilibrium of a string.}\label{rope}
\end{center}
\end{figure}
we sketch a picture of the string whose endpoints are $A$ and $B$ and whose position is described by a function $-s(x)<0$, the origin being below $B$.
The horizontal direction represents the abscissa $x$ whereas the downwards axis represents the vertical displacement.
Denote by $\theta=\theta(x)$ the angle between the horizontal $x$-direction and the tangent to the curve so that
\neweq{wtheta}
-s'(x)=\tan\theta(x)\, .
\endeq
The horizontal component of the tension is constant, that is,
\neweq{H}
F(x)\cos\theta(x)\equiv H_0>0\, .
\endeq

The classical deflection theory of suspension bridges models the bridge structure as a combination of a string (the sustaining cable)
and a beam (the deck). The cable carries its own weight, the weight of the hangers (which we assume
to be negligible) and the weight of the beam. Possible deformations of the cable and the beam are assumed to be small, that is, we aim to model
\neweq{smallw}
\mbox{small displacements of the beam and the cable.}
\endeq
In Figure \ref{cablebeam} we represent the position $-s(x)$ of the cable at rest with $-s_0<0$ being the level of the left and right
endpoints of the cable ($s_0$ is the height of the towers). We denote by $L$ the distance between the towers. Assume that a beam of length $L$ and linear density of
mass $M$ is hanged to the cable whose linear density of mass is $m$: the segment $OQ$ represents the horizontal position of the beam at rest.
Since the spacing between hangers is small relative to the span, the hangers can be considered as a continuous membrane
connecting the cable and the beam. Then the cable is subject to a downwards vertical force given by
$$q(x)=\Big(M+m\sqrt{1+s'(x)^2}\Big)\, g$$
$g$ being the gravitational constant. In this situation, the vertical component of the tension has a variation given by
$$
\frac{d}{dx}[F(x)\sin\theta(x)]=-q(x)
$$
where we recall that the positive vertical axis is oriented downwards. In view of \eq{H} we then obtain
$$
H_0\frac{d}{dx}[\tan\theta(x)]=-q(x)
$$
Moreover, using \eq{wtheta} and taking into account that the cable
is hanged to two towers of same height $s_0$, we get
\neweq{eqparabola}
\begin{cases}
H_0s''(x)=\Big(M+m\sqrt{1+s'(x)^2}\Big)\, g\,,\\
s(0)=s(L)=s_0
\end{cases}
\endeq
If one neglects the mass of the cable ($m=0$) then the solution is the parabola
\neweq{trueparabola}
s_p(x)=s_0-\frac{Mg}{2H_0}x(L-x)\,,
\endeq
while if one neglects the mass of the beam ($M=0$) the solution is the catenary
$$
s_c(x)=\frac{H_0}{mg}\left[\cosh\left(\frac{mg}{2H_0}(2x-L)\right)-\cosh\left(\frac{mgL}{2H_0}\right)\right]+s_0\,.
$$
We point out that if we take $m=0$ in the above explicit equations and we assume that the sag-span ratio for is $1/12$ (a typical value for actual bridges,
see e.g.\ \cite[\S 15.17]{podolny}), by using \eq{trueparabola} we have
$$\frac{gM}{2H_0}=\frac{2}{3L}$$
and then
$$0<s_p(x)-s_c(x)\le s_p(L/2)-s_c(L/2)\approx 6\cdot10^{-3}\, L\qquad\forall x\in(0,L)\, .$$
Hence, for a typical length of the bridge $L=1km$ the maximum difference between the two configurations is about $6m$.\par
Our purpose is to consider both the masses of the beam and the cable ($m,M>0$). Then \eq{eqparabola} does not have simple explicit solutions.
Moreover, this second order equation has no variational structure and the problem has boundary conditions, therefore
  existence and uniqueness of the solution are not granted, but are proved in the next statement.

\begin{proposition}\label{wellposed}
For any $m,M>0$ there exists a unique solution $s=s(x)$ of \eqref{eqparabola}; this solution is symmetric with respect to $x=L/2$, that is, $s(x)=s(L-x)$.
\end{proposition}
\begin{proof} We first prove the existence and uniqueness of a symmetric solution of \eq{eqparabola}, then we show that any solution of \eq{eqparabola}
is symmetric. These two steps will prove the statement.\par
For any $s_1\in\R$ the initial value problem
\neweq{ivp3}
\begin{cases}
H_0s''(x)=\Big(M+m\sqrt{1+s'(x)^2}\Big)\, g\,,\\
s(L/2)=s_1\,,\\
s'(L/2)=0\,,
\end{cases}
\endeq
has a unique solution defined on the whole real line for all values of $s_1$. Moreover, solutions of \eq{ivp3} with different
values of $s_1$ differ by a constant and they are all symmetric with respect to $x=L/2$. Then there exists a unique $s_1<s_0$ such that the solution of
\eq{ivp3} satisfies $s(0)=s(L)=s_0$: this solution is the unique symmetric solution of \eq{eqparabola}.\par
Take now a solution $\overline{s}$ of \eq{eqparabola} which we know to exist in view of what we just proved. Since $\overline{s}$ is convex and
$\overline{s}(0)=\overline{s}(L)$, by the Lagrange Theorem there exists a unique $x_0\in(0,L)$ such that $\overline{s}'(x_0)=0$. Hence,
$\overline{s}$ solves the initial value problem
\neweq{ivp2}
\begin{cases}
H_0s''(x)=\Big(M+m\sqrt{1+s'(x)^2}\Big)\, g\,,\\
s(x_0)=\bar s(x_0)\,,\\
s'(x_0)=0\,.
\end{cases}
\endeq
But also $\overline{s}(2x_0-x)$ solves \eq{ivp2}: then by uniqueness of the solution we infer that $\overline{s}(x)\equiv\overline{s}(2x_0-x)$. Since
$\overline{s}(0)=\overline{s}(L)=s_0$ and since $\overline{s}$ is convex, we finally deduce that $2x_0=L$. Therefore, any solution
$\overline{s}$ of \eq{eqparabola} satisfies $\overline{s}(x)=\overline{s}(L-x)$.\end{proof}

Hence, for general $m,M>0$ the cable takes the shape of a curve with a $\cup$-shaped graph, something in between a parabola and a catenary.
The length at rest of the cable is
\neweq{length}
L_c=\int_0^L\sqrt{1+s'(x)^2}\,dx\,.
\endeq
Let
\neweq{xi}
\sr:=\sqrt{1+s'(x)^2}
\endeq
be the local length of the cable, which is related to the unique solution of \eq{eqparabola} (see Proposition
\ref{wellposed}) and will often appear in the computations. In classical literature this term is usually approximated with
$\sr\equiv1$. While modeling suspension bridges, Biot and von K\'arm\'an \cite[p.277]{biot} warn the reader by writing
{\em ...whereas the deflection of the beam may be considered small, the deflection of the string, i.e., the deviation of its shape from a straight line,
has to be considered as of finite magnitude.} Whence, \eq{smallw} is acceptable whereas no approximation of $s'(x)$ should be taken.
Nevertheless, at a subsequent stage, Biot-von K\'arm\'an \cite[(5.14)]{biot} decide to {\em neglect $s'(x)^2$ in comparison with unity}: this leads to
$\sr\equiv1$. As shown in \cite{gjs} this mistake may significantly change the responses of the system. Needless to say, this rough approximation
considerably simplifies all the computations.

\subsection{Lagrangian of the cable-hangers-beam system}\label{Lagrangebeam}

In this section we write all the components of the Lagrangian of the single cable-hangers-beam system. Since the energies involved are complicated,
we will use \eq{smallw} in order to approximate them with asymptotic expansions up to second order terms. This will give rise to Euler-Lagrange
equations with linear terms.\par\medskip\noindent
$\bullet$ {\bf Kinetic energy of the beam.} If $w=w(x,t)$ denotes the position of the beam, then the kinetic energy of the beam is given by
\neweq{kinetikbeam}
E_{kb}=\frac{M}{2}\int_0^L\dot w^2\, dx\, .
\endeq
$\bullet$ {\bf Kinetic energy of the cable.} If $p(x,t)-s(x)$ denotes the position of the cable, then the kinetic energy of the cable is given by
\neweq{kinetikcable}
E_{kc}=\frac{m}{2}\int_0^L\dot p^2\sr\, dx\, .
\endeq
$\bullet$ {\bf Gravitational energy.} If we neglect the small extensions of the cable, the gravitational energy is given by
$$
E_g=-g\int_0^L\Big(Mw+mp\sr\Big)\,dx\,.
$$
The minus sign is due to the downward orientation of the vertical axis.\par\noindent
$\bullet$ {\bf Elastic energy in the hangers.} The hangers behave as stiff linear springs when in tension and they do not apply any force when compressed.
The latter case is called {\em slackening of the hangers}. Let $\lambda=\lambda(x)$ be the length of the unloaded hanger at position $x\in(0,L)$; this is
the vertical length of the hanger when it is fixed to the sustaining cable and {\em before} hanging the beam. After the beam is installed, the hangers are
in tension and reach a new length $s(x)>\lambda(x)$, where $s$ is the solution of \eq{eqparabola}; if no additional load acts on the system, the equilibrium
position of the beam is $w(x)\equiv0$ (corresponding to the segment $OQ$ in Figure \ref{cablebeam}) and the position of the cable at equilibrium (beam
installed) is $-s(x)$. As one expects from a linear spring,
the Hooke constant $\kappa(x)$ of each hanger is proportional to the inverse of its unloaded length, so that
$$
\kappa(x)=\frac{\kappa_0}{\lambda(x)}\qquad(\kappa_0>0)\,.
$$
Then the (linear density of) force due to the hanger in position $x$ is
\neweq{pospart}
\kappa(x)\Big(s(x)-\lambda(x)\Big)\,.
\endeq
Since the (linear density of) weight of the beam is $Mg$, we find
\neweq{gMk}
Mg=\kappa(x)\Big(s(x)-\lambda(x)\Big)=\frac{\kappa_0\Big(s(x)-\lambda(x)\Big)}{\lambda(x)}
\endeq
which shows that the relative elongation of the hangers is proportional to the density of weight after the beam is installed: this relationship readily
gives the following dependence of the elongation $s(x)$ on the unloaded length
$$
s(x)=\left(1+\frac{Mg}{\kappa_0}\right)\lambda(x)\, .
$$
The actual length of the hanger is $w(x,t)-(p(x,t)-s(x))$, so that the hanger applies a force
\begin{eqnarray}
F(w-p) &=& \kappa(x)\Big(w(x,t)-(p(x,t)-s(x))-\lambda(x)\Big)^+=\kappa(x)\Big(w(x,t)-p(x,t)+\frac{Mg}{\kappa(x)}\Big)^+ \notag\\
\ &=& \Big((\kappa_0+Mg)\frac{w(x,t)-p(x,t)}{s(x)}+Mg\Big)^+\,;\label{force}
\end{eqnarray}
note that $F(0)=Mg$, that is, the restoring force of the hangers balances gravity at equilibrium. The positive part in \eq{force} models the fact that
the hangers behave as linear springs when extended and do not yield any force when slackened. The equalities in \eq{force} are due to \eq{gMk}.
The elastic energy is then given by
\begin{eqnarray}
E_h(w-p) &=& \frac12 \int_0^L \kappa(x)\bigg(\Big(w(x,t)-p(x,t)+\frac{Mg}{\kappa(x)}\Big)^+\bigg)^2\, dx \notag\\
\ &=&\frac12 \int_0^L \kappa(x)\bigg(\Big(w(x,t)-p(x,t)+s(x)-\lambda(x)\Big)^+\bigg)^2\, dx\,.\label{elenhan}
\end{eqnarray}

\begin{remark}
{\em The special form of $F$ in \eq{force} was first suggested by McKenna-Walter \cite{McKennaWalter} and well models the possible slackening of the
hangers, at least in a first order approximation as in the present paper. Brownjohn \cite{brown} claims that the slackening mechanism is not as simple
as an on/off force, as described by the positive part of the elongation. Our results remain true if the force $F$ is replaced by any Lipschitz continuous
function.\endproof}
\end{remark}
\noindent
$\bullet$ {\bf Bending energy of the beam.} The bending energy of a beam depends on its curvature, see \cite{biot} and also \cite{gazgruswe}
for a more recent approach and further references. The Young modulus $E$ is a measure of the stiffness of an elastic material and is
defined as the ratio stress/strain. If $Idx$ denotes the moment of inertia of a
cross section of length $dx$, then the constant quantity $EI$ is the flexural rigidity of the beam. The energy necessary to bend
the beam is the square of the curvature times half the flexural rigidity:
\neweq{bendd}
E_B=\frac{EI}{2}\int_0^L\frac{(w'')^{2}}{\left( 1+(w')^{2}\right)^{3}}\sqrt{1+(w')^2}\, dx
\endeq
where we highlighted the curvature and the arclength. This energy is not convex and it fails to be
coercive in any reasonable functional space so that standard methods
of calculus of variations do not apply. Assuming \eq{smallw}, an asymptotic expansion in \eq{bendd} yields
\neweq{benden}
E_B=\frac{EI}{2}\int_0^L\frac{(w'')^{2}}{\left( 1+(w')^{2}\right)^{5/2}}\, dx\approx
\frac{EI}{2}\int_0^L (w'')^2\left(1-\frac52 (w')^{2}\right)\, dx=\frac{EI}{2}\int_0^L (w'')^2\, dx+o(w^3)\, .
\endeq
Therefore, from now on, we simply take
$$
E_B=\frac{EI}{2} \int_0^L (w'')^{2}\, dx\, .
$$
\noindent
$\bullet$ {\bf Stretching energy of the beam.} If large deformations are involved, the strain-displacement relation is not linear and
a possible nonlinear model was suggested by Woinowsky-Krieger \cite{woinowsky}: he modified the classical Bernoulli-Euler beam theory by assuming
a nonlinear dependence of the axial strain on the deformation gradient and by taking into account the stretching of the beam due to its elongation in
the longitudinal direction. In this situation there is a coupling between bending and stretching and the stretching energy is proportional to the
elongation of the beam which results in
$$
E_S=\frac{\gamma}{2}\left(\int_0^L\Big(\sqrt{1+(y')^2}-1\Big)\, dx\right)^2\approx\frac{\gamma}{8}\left(\int_0^L(y')^2\, dx\right)^2
$$
where $\gamma>0$ is the elastic constant of the beam. Since this term is of fourth order, in the sequel we drop it.\par\noindent
$\bullet$ {\bf Stretching energy of the cable.} The tension of the cable consists of two parts, the tension at rest
\neweq{acca0}
H(x)=H_0\sr
\endeq
and the additional tension
$$
\frac{AE}{L_c}\, \Gamma(p)
$$
due to the increment of length $\Gamma(p)$ of the cable. The latter requires the energy
\neweq{Etc}
E_{tc}(p)=\frac{AE}{2L_c}\,\Gamma(p)^2\,.
\endeq

Assuming \eq{smallw}, we may approximate $E_{tc}(p)$ as follows. First, we note the asymptotic expansion
\neweq{asymptotic}
\sqrt{1+[s'(x)-p'(x)]^2}-\sqrt{1+s'(x)^2}=-\frac{s'(x)p'(x)}{\sr}+\frac{p'(x)^2}{2\sr^3}+o\Big(p'(x)^2\Big)\, .
\endeq
Then we approximate the increment $\Gamma(p)$ of the length of the cable, due to its displacement $p$, by using \eq{asymptotic}:
$$\Gamma(p):=\int_0^L\sqrt{1+[s'(x)-p'(x)]^2}\,dx-L_c\approx\int_0^L\left(\frac{p'(x)^2}{2\sr^3}-\frac{s'(x)p'(x)}{\sr}\right)\, dx\, .$$
By squaring we find
$$\Gamma(p)^2=\left(\int_0^L \frac{s'(x)p'(x)}{\sr}\, dx\right)^2+o\Big((p')^2\Big)\, .$$
Therefore, we approximate \eq{Etc} by
$$E_{tc}(p)=\frac{AE}{2L_c}\left(\int_0^L \frac{s'(x)p'(x)}{\sr}\,dx\right)^2\, .$$
By taking the variation of this energy and integrating by parts we obtain
\neweq{deriv}
dE_{tc} (p)z = -\frac{AE}{L_c}\left(\int_0^L \frac{s'(x)p'(x)}{\sr}\, dx\right)\int_0^L \frac{s''(x)z(x)}{\sr^3}\, dx+o(p)\, .
\endeq

On the other hand, the amount of energy needed to deform the cable at rest under the tension \eq{acca0} in the infinitesimal interval $[x,x+dx]$
from the original position $-s(x)$ to $-s(x)+p(x)$ is the variation of length times the tension, that is
$$
E(x)\,dx=H_0\sr\, \Big(\sqrt{1+[s'(x)-p'(x)]^2}-\sqrt{1+s'(x)^2}\Big)\, dx\,,
$$
hence, using \eq{asymptotic} and neglecting the higher order terms, the energy necessary to deform the whole cable is
\neweq{defenergy}
E(p)=\int_0^L E(x)\,dx=H_0\int_0^L\left(-s'(x)p'(x)+\frac{p'(x)^2}{2\sr^2}\right)\,dx\,.
\endeq
Then, the variation of energy $\sigma^{-1}[E(p+\sigma z)-E(p)]$ as $\sigma\to0$ and some integration by parts lead to
\begin{eqnarray}
dE(p)z &=& H_0\int_0^L\bigg(-s'(x)z'(x)+\frac{p'(x)z'(x)}{\sr^2}\bigg)\,dx \notag\\
\ &=& H_0\int_0^L\bigg(s''-\frac{p''}{\sr^2}+\frac{2s's''p'}{\sr^4}\bigg)z(x)\,dx \notag\\
\ &=& H_0\int_0^L\bigg(s''-\frac{p''}{\sr^2}+\frac{2s's''p'}{\sr^4}\bigg)z(x)\,dx\,.\label{H0}
\end{eqnarray}
\noindent
$\bullet$ {\bf The Euler-Lagrange equations.} The Lagrangian of the
system is obtained as the sum of the kinetic energies minus the sum of
the potential energies. To compute the Euler-Lagrange equations we derive the term
\neweq{acca2}
H_0\bigg(s''(x)-\frac{p''(x)}{\sr^2}+\frac{2s'(x)s''(x)p'(x)}{\sr^4}\bigg)
=\Big(M+m\sr\Big)\, g-\frac{H_0}{\sr^2}p''(x)+\frac{2H_0\, s'(x)s''(x)}{\sr^4}p'(x)
\endeq
from \eq{H0} and using \eq{eqparabola}. Then, from \eq{deriv} we derive the additional tension in the cable $h(p)$ produced by its displacement $p$:
\neweq{accadip}
h(p)=-\frac{AE}{L_c}\left[\int_0^L\frac{s'(x)p'(x)}{\sr}dx\right]\, \frac{s''(x)}{\sr^3}\, .
\endeq

\begin{remark} {\em As already mentioned, in classical engineering literature \cite{biot,melan} one usually replaces $\sr=\sqrt{1+s'(x)^2}$ with
$1$ and neglects the mass of the cable ($m=0$) so that one finds \eq{trueparabola}, that is, $s''(x)=\frac{Mg}{H_0}$; with these
rough approximations, one can easily obtain also the {\it second order expansion} of $h(p)$, that is,
\begin{eqnarray*}
\ & & \frac{AE}{2L_c}\left(\int_0^L p'(x)^2\,dx\right)s''(x)+\frac{AE}{L_c}\left(\int_0^L s'(x)p'(x)\, dx\right)(p''(x)-s''(x)-3s'(x)s''(x)p'(x))\\
\ & = & \frac{AE}{2L_c}\left(\int_0^L p'(x)^2\,dx\right)s''(x)-\frac{AE}{L_c}\frac{Mg}{H_0}\left(\int_0^L p(x)\, dx\right)(p''(x)-s''(x)-3s'(x)s''(x)p'(x))
\end{eqnarray*}
where we only see the terms of the expansion up to order $2$. However, in literature one finds several further simplifications.
Timoshenko \cite{timo1} (see also \cite[Chapter 11]{timoshenkobridge}) manipulates the functional $\Gamma(p)$ and reaches the expression
$$\frac{AE}{L_c}\left(\int_0^L \left(\frac{Mg}{H_0}\, p(x)+\frac{p'(x)^2}{2}\right)\,dx\right)(s''(x)-p''(x))$$
see \cite[(11.16)]{timoshenkobridge}. This expression fails to consider some second order terms but inserts some third order terms. Further
simplifications can be found in literature; Biot-von K\'arm\'an \cite[(5.14)]{biot} neglect the second term and simply obtain
$$\frac{AE}{L_c}\frac{Mg}{H_0}\left(\int_0^L p(x)\,dx\right)(s''(x)-p''(x))\, .$$
This clearly simplifies many computations but this approximation seems not to have a sound justification. This is one further reason to stick to
a {\it first order approximation} of $h(p)$.\endproof}\end{remark}

Let $F$ be the force in \eq{force}, then by taking into account all the energy variations previously found we obtain the following equations of
motion:
\begin{eqnarray}
m\sr\ddot p&=& \frac{H_0}{\sr^2}p''(x)-\frac{2H_0\, s'(x)s''(x)}{\sr^4}p'(x)-h(p)+F(w-p)-Mg\,, \label{onecable}\\
M\ddot w&=& -EIw''''-F(w-p)+Mg\,. \label{onedeck}
\end{eqnarray}
We emphasize that the term $mg\sr$ has disappeared in \eq{onecable} due to its appearance with the opposite sign in \eq{acca2} and in the gravitational energy.

\subsection{A model for the suspension bridge}\label{energybalance}

In this subsection we consider the full suspension bridge. We view the deck as a degenerate plate, namely a beam representing the midline
of the roadway with cross sections which are free to rotate around the
beam. The midline corresponds to the barycenter of the cross sections, which are seen as rods that can rotate around
their barycenter, the angle of rotation with respect to the horizontal position is denoted by $\theta$. The endpoints of the cross sections are the edges of
the plate and they are connected to the sustaining cables through the hangers.
Then, for small $\theta$, the positions of the two free edges of the deck are given by
$$w=y\pm\ell\sin\theta\approx y\pm\ell\theta\, .$$

Each free edge of the deck is connected to a sustaining cable through hangers. We derive here the equation of
this cable-hangers-deck system and, in order to derive the Euler-Lagrange equations, we start from the
cable-hangers-beam model developed in the previous subsections.\par
We denote by $p_1(x,t)$ and $p_2(x,t)$ the displacements of the cables from their equilibrium position $s(x)$, see \eq{eqparabola}:
hence, $p_i(x,t)-s(x)$ denotes the actual position of the $i$-th cable ($i=1,2$). We assume that the roadway has length $L$ and width
$2\ell$ with $2\ell\ll L$.\par\noindent
$\bullet$ {\bf Kinetic energy of the deck.} This energy is composed by two terms, the first corresponding to the kinetic energy of the barycenter
of the cross section, see \eq{kinetikbeam}, the second corresponding to the kinetic energy of the torsional angle.
The kinetic energy of a rotating object is $\frac12J\dot{\theta}^2$, where $J$ is the moment of inertia and
$\dot{\theta}$ is the angular velocity. The moment of inertia of a rod of length $2\ell$ about the perpendicular axis through its center
is given by $\frac{1}{3}M\ell^2dx$ where $Mdx$ is the mass of the rod. Hence, the kinetic energy of a rod having mass $Mdx$ and half-length
$\ell$, rotating about its center with angular velocity $\dot{\theta}$, is given by $\frac{Mdx}{6}\ell^2\dot{\theta}^2$.
Therefore, the kinetic energy of the deck is given by
\neweq{kineticdeck}
E_{kd}=\int_0^L\left(\frac{M}{6}\ell^2\dot\theta^2+\frac{M}{2}\dot y^2\right)\, dx\, .
\endeq
Note that now each cable sustains the weight of half deck, so the boundary value problem \eq{eqparabola} for $s(x)$ becomes instead
$$
H_0s''(x)=\left(\frac{M}{2}+m\sr\right)\, g\,,\qquad s(0)=s(L)=s_0\,.
$$
\noindent
$\bullet$ {\bf Kinetic energy of the cables.} We adopt \eq{kinetikcable} for both cables and we get
\neweq{kineticclables}
E_{kc}=\frac{m}{2}\int_0^L(\dot p_1^2+\dot p_2^2)\, \sr\, dx\, .
\endeq\noindent
$\bullet$ {\bf Stiffening energy of the deck.} It is composed by two terms, the bending energy of the central beam and the torsional energy.
As in \eq{benden}, we have
\neweq{stiffdeck}
E_B=\frac{EI}{2} \int_0^L (y'')^{2}\, dx\, ,
\endeq
whereas the torsional stiffness of the deck is computed in terms of the derivative of the torsional angle:
$$
E_T=\frac{GK}{2} \int_0^L(\theta')^{2}\, dx\, .
$$
\noindent
$\bullet$ {\bf Gravitational energy of the deck and cables.} These are readily computed and read
$$
-Mg\int_0^L y\,dx\qquad\mbox{and}\qquad-mg\int_0^L (p_1+p_2)\sr\,dx\,.
$$
$\bullet$ {\bf Elastic energy in the hangers.} Following \eq{elenhan}, the energies of the two rows of hangers are
$$
E_{h_1}(y+\ell\theta-p_1)=\frac12 \int_0^L \kappa(x)\bigg(\Big(y(x,t)+\ell\theta(x,t)-p_1(x,t)+s(x)-\lambda(x)\Big)^+\bigg)^2\, dx\,,
$$
$$
E_{h_2}(y-\ell\theta-p_2)=\frac12 \int_0^L \kappa(x)\bigg(\Big(y(x,t)-\ell\theta(x,t)-p_2(x,t)+s(x)-\lambda(x)\Big)^+\bigg)^2\, dx\,.
$$
Since the weight $Mg$ of the deck is now supported by two cables, each one supporting $Mg/2$, the force $F$ in \eq{force} acts now on the two edges of the deck
and this gives rise to the two forces:
\begin{eqnarray*}
F(y+\ell\theta-p_1) &=& \kappa(x)\Big(y(x,t)+\ell\theta(x,t)-p_1(x,t)+s(x)-\lambda(x)\Big)^+\\
\ &=& \kappa(x)\Big(y(x,t)+\ell\theta(x,t)-p_1(x,t)+\frac{Mg}{2\kappa(x)}\Big)^+\, ,\\
F(y-\ell\theta-p_2) &=& \kappa(x)\Big(y(x,t)-\ell\theta(x,t)-p_2(x,t)+s(x)-\lambda(x)\Big)^+\\
\ &=& \kappa(x)\Big(y(x,t)-\ell\theta(x,t)-p_2(x,t)+\frac{Mg}{2\kappa(x)}\Big)^+\, ,
\end{eqnarray*}
with $F(0)=Mg/2$. Moreover, instead of \eq{gMk} we have
$$\frac{Mg}{2}=\kappa(x)\Big(s(x)-\lambda(x)\Big)=\frac{\kappa_0\Big(s(x)-\lambda(x)\Big)}{\lambda(x)}\, .$$
\par\medskip\noindent
$\bullet$ {\bf The Euler-Lagrange equations.} Let $h(p)$ be as in \eq{accadip}, then the equations of motion are
\begin{align}
\label{cable1} m\sr\, \ddot p_1&=\frac{H_0}{\sr^2}p_1''-\frac{2H_0\, s'(x)s''(x)}{\sr^4}p_1'-h(p_1)+F(y+\ell\theta-p_1)-\frac{Mg}{2}\\
\label{cable2} m\sr\, \ddot p_2&=\frac{H_0}{\sr^2}p_2''-\frac{2H_0\, s'(x)s''(x)}{\sr^4}p_2'-h(p_2)+F(y-\ell\theta-p_2)-\frac{Mg}{2}\\
\label{decky}M\ddot y&=-EIy''''-F(y+\ell\theta-p_1)-F(y-\ell\theta-p_2)+Mg\\
\label{decktheta}\frac{M}{3}\ell^2\ddot \theta&=GK\theta''+\ell\, F(y-\ell\theta-p_2)-\ell\, F(y+\ell\theta-p_1).
\end{align}
$\bullet$ {\bf Boundary conditions.} The degenerate plate is assumed to be hinged between the two towers while it is clear that the cross
sections between the towers are fixed and cannot rotate. This results in the boundary conditions
\neweq{bc}
y(0,t)=y(L,t)=y''(0,t)=y''(L,t)=\theta(0,t)=\theta(L,t)=0\quad\forall t\ge0\, .
\endeq
Since $p_i-s$ denotes the actual position of the cables, we also have
\neweq{bc2}
p_i(0,t)=p_i(L,t)=0\quad\forall t\ge0\, ,\qquad(i=1,2).
\endeq

\section{Existence and uniqueness for the suspension bridge system}\label{wellpos}

\subsection{Definition of solution}\label{basi}

Up to scaling, we may assume that $L=\pi$: this will simplify the spectral decomposition of the differential operator.
We endow the Hilbert spaces $L^2(0,\pi)$, $H^1_0(0,\pi)$, $H^2\cap H^1_0(0,\pi)$ with, respectively, the scalar products
$$(u,v)_2=\int_0^\pi uv\ ,\quad(u,v)_{H^1}=\int_0^\pi u'v'\ ,\quad(u,v)_{H^2}=\int_0^\pi u''v''\ .$$
Then an orthogonal basis in these three spaces is given by $\{e_k\}_{k=1}^\infty$, where
$$e_k(x)=\sqrt{\frac2\pi}\, \sin(kx)\, ,\quad\|e_k\|_2=1\, ,\quad\|e_k\|_{H^1}=k\, ,\quad\|e_k\|_{H^2}=k^2\, .$$
We denote by $H^*(0,\pi)$ the dual space of $H^2\cap H^1_0(0,\pi)$ and by $\langle\cdot,\cdot\rangle_*$ the duality pairing, whereas we denote by
$H^{-1}(0,\pi)$ the dual space of $H^1_0(0,\pi)$ and by $\langle\cdot,\cdot\rangle_1$ the duality pairing.\par
Let $\xi$ be as in \eq{xi}; we also consider the Sturm-Liouville eigenvalue problem
\neweq{eigenxi}
-\left(\frac{H_0}{\sr^2}\, u'\right)'=\lambda \sr u\quad\mbox{ in }(0,\pi)\, ,\qquad u(0)=u(\pi)=0\, .
\endeq
We say that $\lambda$ is an eigenvalue of \eq{eigenxi} if there exists $u\neq0$ (eigenfunction) such that
$$\int_0^\pi \frac{H_0}{\sr^2}\, u'v'=\lambda\int_0^\pi \sr uv\qquad\forall v\in H^1_0(0,\pi)\, .$$
It is well-known (see e.g.\ \cite[\S 2.4]{algwaiz}) that \eq{eigenxi} admits a sequence of positive eigenvalues $\{\lambda_k\}$ which diverges to infinity.
Each eigenvalue has multiplicity 1 and the sequence of eigenfunctions $\{u_k\}$ is an orthogonal basis of $L^2(0,\pi)$ and of $H^1_0(0,\pi)$ endowed
with the scalar products
$$(u,v)_\xi=\int_0^\pi\sr\, uv\qquad\forall u,v\in L^2_\xi(0,\pi)\, ,\qquad(u,v)_{H^1_\xi}=\int_0^\pi \frac{H_0}{\sr^2}\, u'v'\qquad
\forall u,v\in H^1_\xi(0,\pi)\, .$$
Let us emphasize that the corresponding norms $\|\cdot\|_\xi$ and $\|\cdot\|_{H^1_\xi}$ are equivalent, respectively, to the norms $\|\cdot\|_2$
and $\|\cdot\|_{H^1}$ but, for the sake of clarity, we maintain the different notations $L^2_\xi$ and $H^1_\xi$. In the sequel, we assume that $\{u_k\}$
is normalized in $L^2_\xi$, that is,
$$\|u_k\|_\xi=1\, ,\quad\|u_k\|_{H^1_\xi}=\sqrt{\lambda_k}\, .$$
We denote by $H^\xi(0,\pi)$ the dual space of $H^1_\xi(0,\pi)$ and by $\langle\cdot,\cdot\rangle_\xi$ the corresponding duality pairing.\par
We set
$$
\Phi(s):=F(s)-\frac{Mg}{2}\qquad\mbox{and}\qquad\Psi(s)=\int_0^s\Phi(\sigma)\, d\sigma
$$
so that
$$
\Phi(0)=\Psi(0)=0\ ,\qquad\Phi(s)\ge0\, ,\ \Psi(s)\ge0\quad \forall s\in\R\, .
$$

In the above functional-analytic setting and with this simplification, the equations \eq{cable1}-\eq{cable2}-\eq{decky}-\eq{decktheta}
may be rewritten as
\neweq{primordine}
\left\{\begin{array}{ll}
m\sr\, \ddot p_1&=\left[\frac{H_0}{\sr^2}p_1'\right]'+\frac{AE}{L_c}\left[\int_0^\pi\frac{s'(x)p_1'(x)}{\sr}dx\right]\frac{s''(x)}{\sr^3}+\Phi(y+\ell\theta-p_1)\\
m\sr\, \ddot p_2&=\left[\frac{H_0}{\sr^2}p_2'\right]'+\frac{AE}{L_c}\left[\int_0^\pi\frac{s'(x)p_2'(x)}{\sr}dx\right]\frac{s''(x)}{\sr^3}+\Phi(y-\ell\theta-p_2)\\
M\ddot y&=-EIy''''-\Phi(y+\ell\theta-p_1)-\Phi(y-\ell\theta-p_2)\\
\frac{M}{3}\ell^2\ddot \theta&=GK\theta''+\ell\Phi(y-\ell\theta-p_2)-\ell\Phi(y+\ell\theta-p_1)
\end{array}\right.
\endeq
where $\sr$ and $s(x)$ are linked through \eq{xi}. We add the boundary conditions \eq{bc}-\eq{bc2} which we rewrite for $L=\pi$:
\neweq{bcpi}
y(0,t)=y(\pi,t)=y''(0,t)=y''(\pi,t)=p_i(0,t)=p_i(\pi,t)=\theta(0,t)=\theta(\pi,t)=0\quad\forall t\ge0\, .
\endeq
We fix some initial data at time $t=0$,
\neweq{initial}
\begin{array}{cc}
y(x,0)=y^0(x)\,,\quad p_i(x,0)=p_i^0(x)\,,\quad \theta(x,0)=\theta^0(x)\qquad\forall x\in(0,\pi)\\
\dot{y}(x,0)=y^1(x)\,,\quad\dot{p}_i(x,0)=p_i^1(x)\,,\quad\dot\theta(x,0)=\theta^1(x)\qquad\forall x\in(0,\pi)\, ,
\end{array}
\endeq
with the following regularity
\neweq{reginitial}
y^0\in H^2\cap H^1_0(0,\pi)\, ,\quad\theta^0,p_i^0\in H^1_0(0,\pi)\, ,\quad y^1,\theta^1,p_i^1\in L^2(0,\pi)\, .
\endeq

We say that $(p_1,p_2,y,\theta)$ is a weak solution of \eq{primordine} if
\neweq{regularity}
\begin{array}{cc}
\displaystyle\theta,p_i\in C^0\Big([0,T];H^1_0(0,\pi)\Big)\cap C^1\Big([0,T];L^2(0,\pi)\Big)\cap C^2\Big([0,T];H^{-1}(0,\pi)\Big)\\
y\in C^0\Big([0,T];H^2\cap H^1_0(0,\pi)\Big)\cap C^1\Big([0,T];L^2(0,\pi)\Big)\cap C^2\Big([0,T];H^*(0,\pi)\Big)
\end{array}\endeq
and if the following equations are satisfied
$$
\left\{\begin{array}{l}
m\langle\ddot p_1,v_1\rangle_\xi+(p_1,v_1)_{H^1_\xi}=\frac{AE}{L_c}\left[\int_0^\pi\frac{s'(x)p_1'(x)}{\sr}dx\right]\left(\frac{s''(x)}{\sr^3},v_1\right)_2
+\Big(\Phi(y+\ell\theta-p_1),v_1\Big)_2\\
m\langle\ddot p_2,v_2\rangle_\xi+(p_2,v_2)_{H^1_\xi}=\frac{AE}{L_c}\left[\int_0^\pi\frac{s'(x)p_2'(x)}{\sr}dx\right]\left(\frac{s''(x)}{\sr^3},v_2\right)_2
+\Big(\Phi(y-\ell\theta-p_2),v_2\Big)_2\\
M\langle\ddot y,w\rangle_*+EI(y,w)_{H^2}=-\Big(\Phi(y+\ell\theta-p_1)+\Phi(y-\ell\theta-p_2),w\Big)_2\\
\frac{M}{3}\ell^2\langle\ddot \theta,v_3\rangle_1+GK(\theta,v_3)_{H^1}=
\Big(\ell\Phi(y-\ell\theta-p_2)-\ell\Phi(y+\ell\theta-p_1),v_3\Big)_2
\end{array}\right.
$$
for all $v_i\in H^1_0(0,\pi)$, $w\in H^2\cap H^1_0(0,\pi)$ and $t>0$. In the sequel we denote by $X_T$ the functional space of solutions:
\neweq{X}
(p_1,p_2,y,\theta)\in X_T\ \Longleftrightarrow\ \mbox{\eq{regularity} holds.}
\endeq
Note that this space already includes the boundary conditions \eq{bcpi}: hence, from now on we will not mention further \eq{bcpi}.\par
In view of the energy balance performed in Section \ref{energybalance}, the conserved (and approximated at second order) total energy
of any solution $(p_1,p_2,y,\theta)\in X_T$ of \eq{primordine} is given by
\begin{eqnarray}
E(t) &=& \int_0^\pi\left(\frac{M}{6}\ell^2\dot\theta^2+\frac{M}{2}\dot y^2+\frac{m}{2}(\dot p_1^2+\dot p_2^2)\, \sr\right)\, dx
+\int_0^\pi\left(\frac{EI}{2}(y'')^{2}+\frac{GK}{2}(\theta')^{2}\right)\, dx \notag \\
\ & \ &+\frac{AE}{2L_c}\left(\int_0^\pi \frac{s'\, p_1'}{\sr}\,dx\right)^2+\frac{AE}{2L_c}\left(\int_0^\pi \frac{s'\, p_2'}{\sr}\,dx\right)^2
+H_0\int_0^\pi\left(\frac{(p_1')^2+(p_2')^2}{2\sr^2}-s'\, (p_1'+p_2')\right)\,dx \notag \\
\ & \ & -mg\int_0^\pi(p_1+p_2)\sr\,dx+\int_0^\pi\Big(\Psi(y+\ell\theta-p_1)+\Psi(y-\ell\theta-p_2)\Big)\, dx \label{conserved}
\end{eqnarray}
where, in the first line we see the total kinetic energy of the bridge and the elastic energy of the deck whereas in the second line we see the stretching
energy of the two cables. The third line in \eq{conserved} deserves a particular attention: we see there the gravitational energies of the
cables and the deck, the latter being included in $\Psi$ while the former cancels in the equations \eq{primordine} due to its presence with an opposite
sign in both \eq{acca2} and in the gravitational energy.

\subsection{Existence and uniqueness for a related linear problem}

We consider here the linear and decoupled problem
\neweq{lineare}
\left\{\begin{array}{ll}
m\sr\, \ddot p_1&=\left[\frac{H_0}{\sr^2}p_1'\right]'+g_1(x,t)\\
m\sr\, \ddot p_2&=\left[\frac{H_0}{\sr^2}p_2'\right]'+g_2(x,t)\\
M\ddot y&=-EIy''''+g_3(x,t)\\
\frac{M}{3}\ell^2\ddot \theta&=GK\theta''+g_4(x,t)
\end{array}\right.\qquad\qquad x\in(0,\pi)\, ,\ t>0\, .
\endeq
where $g_j\in C^0([0,\pi]\times[0,T])$ for $j=1,...,4$. To \eq{lineare} we associate the initial conditions \eq{initial} with the regularity
as in \eq{reginitial}. We say that $(p_1,p_2,y,\theta)\in X_T$ is a weak solution of \eq{lineare} if
$$
\left\{\begin{array}{l}
m\langle\ddot p_1,v\rangle_\xi+(p_1,v)_{H^1_\xi}= (g_1,v)_2\\
m\langle\ddot p_2,v\rangle_\xi+(p_2,v)_{H^1_\xi}= (g_2,v)_2\\
M\langle\ddot y,w\rangle_*+EI(y,w)_{H^2}=(g_3,w)_2\\
\frac{M}{3}\ell^2\langle\ddot \theta,v\rangle_1+GK(\theta,v)_{H^1}=(g_4,v)_2
\end{array}\right.\qquad\forall v\in H^1_0(0,\pi)\, ,\ \forall w\in H^2\cap H^1_0(0,\pi)\, ,\ t>0\, .
$$

The purpose of the present section is to prove the following statement.

\begin{theorem}\label{wplinear}
Let $T>0$ and let $g_j\in C^0([0,\pi]\times[0,T])$ for $j=1,...,4$. For all $y^0,\theta^0,p_i^0,y^1,\theta^1,p_i^1$ satisfying \eqref{reginitial} there
exists a unique weak solution $(p_1,p_2,y,\theta)\in X_T$ of \eqref{lineare} satisfying \eqref{initial}.
\end{theorem}

The proof of Theorem \ref{wplinear} uses a Galerkin procedure and is divided in several steps; it allows to emphasize the delicate role of all the
spaces, norms, and scalar products defined in Section \ref{basi}. Moreover, the proof provides the underlying idea for a numerical approximation
of the dynamics with a finite number of modes.\par\noindent
{\bf \underline{Step 1}.} We construct a sequence of solutions of approximated problems in finite dimensional spaces.\par
We consider the two basis defined in Section \ref{basi} and, for any $n\ge1$, we put
$$E_n:={\rm span }\{e_1,\dots,e_n\}\, ,\qquad U_n:={\rm span }\{u_1,\dots,u_n\}\, .$$
For any $n\ge1$ we also put
$$(p_i^0)_n:=\sum_{k=1}^n (p_i^0,u_k)_\xi\, u_k=-\sum_{k=1}^n \lambda_k^{-1/2}(p_i^0,u_k)_{H^1_\xi}\, u_k\, ,
\qquad(p_i^1)_n:=\sum_{k=1}^n (p_i^1,u_k)_\xi\, u_k\, ,$$
$$y^0_n:=\sum_{k=1}^n (y^0,e_k)_2\, e_k=\sum_{k=1}^n k^{-2}(y^0,e_k)_{H^2}\, e_k\, ,\qquad y^1_n:=\sum_{k=1}^n (y^1,e_k)_2\, e_k\, ,$$
$$\theta^0_n:=\sum_{k=1}^n (\theta^0,e_k)_2\, e_k=-\sum_{k=1}^n k^{-1}(\theta^0,e_k)_{H^1}\, e_k\, ,\qquad \theta^1_n:=\sum_{k=1}^n (\theta^1,e_k)_2\, e_k\, ,$$
so that
\neweq{convinitial}
y^0_n\to y^0\mbox{ in }H^2\, ,\quad\theta^0_n\to\theta^0\, ,\ (p_i^0)_n\to p_i^0\mbox{ in }H^1\, ,\quad
y^1_n\to y^1\, ,\ \theta^1_n\to\theta^1\, ,\ (p_i^1)_n\to p_i^1\mbox{ in }L^2
\endeq
as $n\to\infty$. For any $n\ge 1$ we seek $((p_1)_n,(p_2)_n,y_n,\theta_n)$ such that
$$(p_i)_n(x,t)=\sum_{k=1}^n (p_i)_n^k(t)u_k\, ,\qquad y_n(x,t)=\sum_{k=1}^n y_n^k(t)e_k\, ,\qquad\theta_n(x,t)=\sum_{k=1}^n \theta_n^k(t)e_k$$
and solving the variational problem
\neweq{approxsyst}
\left\{\begin{array}{l}
m((\ddot p_1)_n,v)_\xi+((p_1)_n,v)_{H^1_\xi}= (g_1,v)_2\\
m((\ddot p_2)_n,v)_\xi+((p_2)_n,v)_{H^1_\xi}= (g_2,v)_2\\
M(\ddot y_n,w)_2+EI(y_n,w)_{H^2}=(g_3,w)_2\\
\frac{M}{3}\ell^2(\ddot \theta_n,w)_2+GK(\theta_n,w)_{H^1}=(g_4,w)_2
\end{array}\right.\qquad\forall v\in U_n\, ,\ \forall w\in E_n\, ,\ t>0\, .
\endeq
By making $n$ tests on each equation (for $v=u_1,...,u_n$ and $w=e_1,...,e_n$) we obtain the $4n$ linear equations
\neweq{ntests}
\left\{\begin{array}{l}
m(\ddot p_1)_n^k+\lambda_k(p_1)_n^k=(g_1,u_k)_2\\
m(\ddot p_2)_n^k+\lambda_k(p_2)_n^k=(g_2,u_k)_2\\
M\ddot y_n^k+EI\, k^2\, y_n^k=(g_3,e_k)_2\\
\frac{M}{3}\ell^2\ddot\theta_n^k+GK\, k\, \theta_n^k=(g_4,e_k)_2
\end{array}\right.\qquad\forall k=1,...,n\, .
\endeq
It is a classical result from the theory of linear ODE's that this finite-dimensional linear system, together with the initial conditions
$$
(p_i)_n^k(0)=(p_i^0,u_k)_\xi\, ,\ (\dot{p}_i)_n^k(0)=(p_i^1,u_k)_\xi\, ,\ y_n^k(0)=(y^0,e_k)_2\, ,\ \dot{y}_n^k(0)=(y^1,e_k)_2\, ,
$$
$$
\theta_n^k(0)=(\theta^0,e_k)_2\, ,\ \dot{\theta}_n^k(0)=(\theta^1,e_k)_2\, ,
$$
admits a unique global solution defined for all $t>0$.\par\noindent
{\bf \underline{Step 2}.} In this step we prove a local uniform bound for the sequence $\{((p_1)_n,(p_2)_n,y_n,\theta_n)\}$.\par
We fix some (finite) $T>0$; in what follows the $c_i$'s denote positive constants independent of $n$, possibly depending on $T$.
We take $w=\dot y_n$ in \eq{approxsyst}$_3$ and we integrate first over $(0,\pi)$ and then over $(0,t)$ for some $t\in(0,T)$ to obtain
\begin{eqnarray}
M\, \|\dot y_n(t)\|_2^2+EI\, \|y_n(t)\|_{H^2}^2 &=& c_1+2\int_0^t(g_3,\dot y_n)_2\le
c_1+2\|g_3\|_\infty\int_0^t\|\dot y_n(\tau)\|_1\, d\tau \notag\\
&\le& c_1+c_2\int_0^t\|\dot y_n(\tau)\|_1\, d\tau \label{stima1}
\end{eqnarray}
where $\|\cdot\|_\infty$ denotes the $L^\infty([0,\pi]\times[0,T])$-norm whereas $\|\cdot\|_1$ denotes the $L^1(0,\pi)$-norm.
Here, $c_1:=\sup_n(M\|y^1_n\|_2^2+EI\|y^0_n\|_{H^2}^2)<\infty$. By using the H\"older inequality, we see that \eq{stima1} implies
\neweq{stima2}
\|\dot y_n(t)\|_2^2\le c_3+ c_4 \left(\int_0^t\|\dot y_n(\tau)\|_2^2\, d\tau\right)^{1/2}\, .
\endeq
In turn, \eq{stima2} may be written as
$$f'(t)\le c_3+c_4 \sqrt{f(t)}\qquad\forall t\in(0,T)\, ,\qquad f(t):=\int_0^t\|\dot y_n(\tau)\|_2^2\, d\tau\, .$$
We now recall a Gronwall-type lemma which can be deduced (e.g.) from \cite{tre} and \cite[Lemma A.5/p.157]{quattro}.

\begin{lemma}\label{gronwall}
Let $f\in C^1(\R_+)$ be such that $f(0)=0$, $0\le f'(t)\le C_1+C_2\, \sqrt{f(t)}$ for all $t\ge0$ (for some $C_1,C_2>0$). Then
$$f(t)\le\frac{(C_1+C_2)^2}{4}\, t^2+(C_1+C_2)\, t\qquad\forall t\ge0\, .$$
\end{lemma}

By applying Lemma \ref{gronwall} to \eq{stima2} and going back to \eq{stima1}, we obtain constants $c_5,c_6>0$ (independent of $n$) such that
\neweq{boundy}
\|\dot y_n(t)\|_2^2+\|y_n(t)\|_{H^2}^2\le c_5+c_6\, t^2\qquad\forall t\in(0,T)\ .
\endeq

Similarly, take $v=(\dot p_1)_n$ in \eq{approxsyst}$_1$ or $v=(\dot p_2)_n$ in \eq{approxsyst}$_2$, to obtain constants $c_7,c_8,c_9,c_{10}>0$
(independent of $n$) such that
$$
\|(\dot p_1)_n(t)\|_\xi^2+\|(p_1)_n(t)\|_{H^1_\xi}^2\le c_7+c_8\, t^2\qquad\forall t\in(0,T)\ ,
$$
$$
\|(\dot p_2)_n(t)\|_\xi^2+\|(p_2)_n(t)\|_{H^1_\xi}^2\le c_9+c_{10}\, t^2\qquad\forall t\in(0,T)\ .
$$
In order to obtain these inequalities one also needs to combine the H\"older inequality with the equivalence of the norms in $L^2$ and $L^2_\xi$.\par
Finally, we take $w=\dot \theta_n$ in \eq{approxsyst}$_4$. Proceeding as above, we obtain constants $c_{11},c_{12}>0$ (independent of $n$) such that
$$
\|\dot \theta_n(t)\|_2^2+\|\theta_n(t)\|_{H^1}^2\le c_{11}+c_{12}\, t^2\qquad\forall t\in(0,T)\ .
$$
{\bf \underline{Step 3}.} We show that $\{((p_1)_n,(p_2)_n,y_n,\theta_n)\}$ admits a strongly convergent subsequence in some norm.\par
Let us consider in detail the sequence $\{y_n\}$, the other components being similar. The bound in \eq{boundy} suggests to prove strong convergence in the space
$$C^0([0,T];H^2_*(0,\pi))\cap C^1([0,T];L^2(0,\pi))\, .$$
The equations in \eq{ntests} show that the components $y_n^k$ do not depend on $n$, that is,
$$y_n(x,t)=\sum_{k=1}^n y^k(t)e_k\, .$$
Take some $m>n\ge1$ and define
$$y_{m,n}(x,t):=y_m(x,t)-y_n(x,t)=\sum_{k=n+1}^m y^k(t)e_k$$
so that, in particular,
$$
y_{m,n}(x,0)=y^0_m-y^0_n\, ,\quad \dot y_{m,n}(x,0)=y^1_m-y^1_n\, .
$$
Consider also the Fourier decomposition of the function $g_3$:
$$g_3(x,t)=\sum_{k=1}^\infty g_3^k(t)e_k\, ;$$
since $g_3\in C^0([0,\pi]\times[0,T])\subset L^2((0,\pi)\times(0,T))$, we know that
\neweq{converge}
\left(\sum_{k=n+1}^\infty g_3^k(t)^2\right)^{1/2}\to0\quad\mbox{as }n\to\infty\qquad\mbox{in }L^2(0,T)\, .
\endeq

Rewrite \eq{approxsyst}$_3$ with $n$ replaced by $m$. Then by taking $w=\dot y_{m,n}(t)$ as a test function, by using the orthogonality of the
system $\{e_k\}$, and by integrating over $(0,t)$ we obtain
\begin{eqnarray}
M\|\dot y_{m,n}(t)\|_2^2+EI\|y_{m,n}(t)\|_{H^2}^2 &=& C_{m,n}+2\int_0^t(g_3,\dot y_{m,n})_2=
C_{m,n}+2\int_0^t\left(\sum_{k=n+1}^m g^k(\tau)e_k,\dot y_{m,n}\right)_2 d\tau \notag\\
&\le& C_{m,n}+2\int_0^T \left(\sum_{k=n+1}^m g^k_3(t)^2\right)^{1/2}\|\dot y_{m,n}(t)\|_2 dt \label{stimacauchy}
\end{eqnarray}
where $C_{m,n}=M\|y^1_m-y^1_n\|_2^2+EI\|y^0_m-y^0_n\|_{H^2}^2$. By \eq{convinitial} we know that $C_{m,n}\to0$ as $m,n\to\infty$; combined with
\eq{converge} and \eq{stimacauchy}, this shows that $\{y_n\}$ is a Cauchy sequence in $C^0([0,T];H^2\cap H^1_0(0,\pi))\cap C^1([0,T];L^2(0,\pi))$.
By completeness of these spaces we conclude that
\neweq{convy}
\begin{array}{l}
\exists y\in C^0([0,T];H^2\cap H^1_0(0,\pi))\cap C^1([0,T];L^2(0,\pi))\mbox{ s.t. }\\
y_n\to y\quad\text{in }C^0([0,T];H^2\cap H^1_0(0,\pi))\cap C^1([0,T];L^2(0,\pi))\quad \text{as } n\to +\infty \, ,
\end{array}\endeq
thereby completing the proof of the claim.\par
For the sequences $\{\theta_n\}$, $\{(p_1)_n\}$, $\{(p_2)_n)\}$ we may proceed similarly and prove that
\neweq{convp}
\begin{array}{l}
\exists p_i\in C^0\Big([0,T];H^1_0(0,\pi)\Big)\cap C^1\Big([0,T];L^2(0,\pi)\Big)\mbox{ s.t. }\\
(p_i)_n\to p_i\quad\text{in }C^0\Big([0,T];H^1_0(0,\pi)\Big)\cap C^1\Big([0,T];L^2(0,\pi)\Big)\quad \text{as } n\to +\infty \, ,
\end{array}\endeq
\neweq{convtheta}
\begin{array}{l}
\exists\theta\in C^0\Big([0,T];H^1_0(0,\pi)\Big)\cap C^1\Big([0,T];L^2(0,\pi)\Big)\mbox{ s.t. }\\
\theta_n\to\theta\quad\text{in }C^0\Big([0,T];H^1_0(0,\pi)\Big)\cap C^1\Big([0,T];L^2(0,\pi)\Big)\quad \text{as } n\to +\infty\, .
\end{array}\endeq
{\bf \underline{Step 4}.} We take the limit in \eqref{approxsyst} and we prove Theorem \ref{wplinear}.\par
By using \eq{convy}-\eq{convp}-\eq{convtheta} in \eqref{approxsyst} we see that $(p_1,p_2,y,\theta)$ is a weak solution of \eq{lineare}
satisfying \eqref{initial}, with the additional regularity
$$
\theta,p_i\in C^2\Big([0,T];H^{-1}(0,\pi)\Big)\, ,\quad y\in C^2\Big([0,T];H^*(0,\pi)\Big)
$$
following from the equations \eq{lineare}. Therefore, for any $T>0$ we have proved the existence of a weak solution $(p_1,p_2,y,\theta)\in X_T$ of
\eq{lineare} over the interval $(0,T)$ and satisfying \eqref{initial}. By arbitrariness of $T>0$ this proves global existence. Finally, arguing by
contradiction and assuming the existence of two solutions, we subtract the two linear equations for the two solutions and we obtain an homogeneous
linear problem; the energy conservation then shows that the (nonnegative) energy is always 0, which proves that the two solutions are, in fact, the same.
This completes the proof of Theorem \ref{wplinear}.

\subsection{The existence and uniqueness result}

The purpose of this final subsection is to prove the following statement.

\begin{theorem}\label{wp}
For all $y^0,\theta^0,p_i^0,y^1,\theta^1,p_i^1$ satisfying \eqref{reginitial} there exists a unique (global in time) weak solution
$(p_1,p_2,y,\theta)\in X_\infty$ of \eqref{primordine} satisfying the initial conditions \eqref{initial}.
\end{theorem}

Global in time means here that we can take any $T>0$, including $T=\infty$; this explains the notation $X_\infty$. For the proof we use a fixed point
procedure combined with an energy estimate. The first step consists in defining the map which will be shown to have a fixed point; in the sequel
we emphasize the dependence on time of the nonlocal term.

\begin{lemma}\label{fixed}
Let $T>0$ and assume that $y^0,\theta^0,p_i^0,y^1,\theta^1,p_i^1$ satisfy \eqref{reginitial}. For all $(q_1,q_2,z,\alpha)\in X_T$ there exists a
unique weak solution $(p_1,p_2,y,\theta)\in X_T$ of the system
\begin{eqnarray*}
m\sr\, \ddot p_1&=&\left[\frac{H_0}{\sr^2}p_1'\right]'+\frac{AE}{L_c}\left[\int_0^\pi\frac{s'(x)q_1'(x,t)}{\sr}dx\right]\frac{s''(x)}{\sr^3}+\Phi(z+\ell\alpha-q_1)\\
m\sr\, \ddot p_2&=&\left[\frac{H_0}{\sr^2}p_2'\right]'+\frac{AE}{L_c}\left[\int_0^\pi\frac{s'(x)q_2'(x,t)}{\sr}dx\right]\frac{s''(x)}{\sr^3}+\Phi(z-\ell\alpha-q_2)\\
M\ddot y&=&-EIy''''-\Phi(z+\ell\alpha-q_1)-\Phi(z-\ell\alpha-q_2)\\
\frac{M}{3}\ell^2\ddot \theta&=&GK\theta''+\ell\Phi(z-\ell\alpha-q_2)-\ell\Phi(z+\ell\alpha-q_1)
\end{eqnarray*}
satisfying the initial conditions \eqref{initial}.
\end{lemma}
\begin{proof} We set
$$g_1(x,t)=\frac{AE}{L_c}\left[\int_0^\pi\frac{s'(x)q_1'(x,t)}{\sr}dx\right]\frac{s''(x)}{\sr^3}+\Phi(z+\ell\alpha-q_1)\, ,$$
$$g_2(x,t)=\frac{AE}{L_c}\left[\int_0^\pi\frac{s'(x)q_2'(x,t)}{\sr}dx\right]\frac{s''(x)}{\sr^3}+\Phi(z-\ell\alpha-q_2)\, ,$$
$$g_3(x,t)=-\Phi(z+\ell\alpha-q_1)-\Phi(z-\ell\alpha-q_2)\, ,$$
$$g_4(x,t)=\ell\Phi(z-\ell\alpha-q_2)-\ell\Phi(z+\ell\alpha-q_1)\, .$$
Then $g_j\in C^0([0,\pi]\times[0,T])$ for $j=1,...,4$ and we apply Theorem \ref{wplinear}.\end{proof}

Denote by $Z_T$ the subspace of $X_T$ such that
\neweq{regularity2}
\begin{array}{cc}
\displaystyle\theta,p_i\in C^0\Big([0,T];H^1_0(0,\pi)\Big)\cap C^1\Big([0,T];L^2(0,\pi)\Big)\\
y\in C^0\Big([0,T];H^2\cap H^1_0(0,\pi)\Big)\cap C^1\Big([0,T];L^2(0,\pi)\Big)
\end{array}\endeq
which is a Banach space when endowed with the norm
\begin{eqnarray*}
\|(p_1,p_2,y,\theta)\|_{Z_T}^2 &:=& \|p_1\|_{L^\infty(H^1_\xi)}^2+\|\dot{p}_1\|_{L^\infty(L^2_\xi)}^2+\|p_2\|_{L^\infty(H^1_\xi)}^2
+\|\dot{p}_2\|_{L^\infty(L^2_\xi)}^2\\
\ &\ & +\|y\|_{L^\infty(H^2)}^2+\|\dot{y}\|_{L^\infty(L^2)}^2+\|\theta\|_{L^\infty(H^1)}^2+\|\dot{\theta}\|_{L^\infty(L^2)}^2
\end{eqnarray*}
where the norms $L^\infty(\cdot)$ have the usual time-space meaning. For any $y^0,\theta^0,p_i^0,y^1,\theta^1,p_i^1$ satisfying \eq{reginitial}
we define the convex subset (complete metric space)
$$\B:=\{(p_1,p_2,y,\theta)\in Z_T;\ \mbox{\eq{initial} holds}\}\, .$$
If we take $(q_1,q_2,z,\alpha)$ and $(p_1,p_2,y,\theta)$ both satisfying \eq{initial}, then Lemma \ref{fixed} defines a map
\neweq{Upsilon}
\Upsilon:\B\to\B\, ,\qquad  \Upsilon(q_1,q_2,z,\alpha)=(p_1,p_2,y,\theta)\, .
\endeq
We now prove that for sufficiently small $T$ this map is contractive.

\begin{lemma}\label{contraction}
Let $y^0,\theta^0,p_i^0,y^1,\theta^1,p_i^1$ satisfy \eqref{reginitial}. If $T>0$ is sufficiently small, then the map $\Upsilon$ defined in
\eqref{Upsilon} is a contraction from $\B$ into $\B$.
\end{lemma}
\begin{proof} Consider two different $(q_1^1,q_2^1,z^1,\alpha^1)$ and $(q_1^2,q_2^2,z^2,\alpha^2)$ in $\B$ and let
$(p_1^j,p_2^j,y^j,\theta^j)=\Upsilon(q_1^j,q_2^j,z^j,\alpha^j)$ for $j=1,2$. Denote by
$$p_1=p_1^1-p_1^2\, ,\quad p_2=p_2^1-p_2^2\, ,\quad y=y^1-y^2\, ,\quad \theta=\theta^1-\theta^2\, .$$
We underline that these notations should not be confused with the initial conditions \eq{initial}.
By subtracting the two systems satisfied by $(p_1^j,p_2^j,y^j,\theta^j)$ we see that $(p_1,p_2,y,\theta)$ is a weak solution of
\begin{eqnarray}
m\sr\, \ddot p_1&=&\left[\frac{H_0}{\sr^2}p_1'\right]'+
\frac{AE}{L_c}\left[\int_0^\pi\frac{s'(x)[(q_1^1)'(x,t)-(q_1^2)'(x,t)]}{\sr}dx\right]\frac{s''(x)}{\sr^3} \notag\\
 & & +\Phi(z^1+\ell\alpha^1-q_1^1)-\Phi(z^2+\ell\alpha^2-q_1^2) \label{diffp1}
\end{eqnarray}
\begin{eqnarray}
m\sr\, \ddot p_2&=& \left[\frac{H_0}{\sr^2}p_2'\right]'+
\frac{AE}{L_c}\left[\int_0^\pi\frac{s'(x)[(q_2^1)'(x,t)-(q_2^2)'(x,t)]}{\sr}dx\right]\frac{s''(x)}{\sr^3} \notag\\
& &+\Phi(z^1-\ell\alpha^1-q_2^1)-\Phi(z^2-\ell\alpha^2-q_2^2) \label{diffp2}
\end{eqnarray}
\begin{eqnarray}
M\ddot y&=&-EIy''''-\Phi(z^1+\ell\alpha^1-q_1^1)-\Phi(z^1-\ell\alpha^1-q_2^1) \notag\\
& &+\Phi(z^2+\ell\alpha^2-q_1^2)+\Phi(z^2-\ell\alpha^2-q_2^2) \label{diffy}
\end{eqnarray}
\begin{eqnarray}
\frac{M}{3}\ell^2\ddot \theta&=&GK\theta''+\ell[\Phi(z^1-\ell\alpha^1-q_2^1)-\Phi(z^1+\ell\alpha^1-q_1^1)] \notag\\
& &-\ell[\Phi(z^2-\ell\alpha^2-q_2^2)-\Phi(z^2+\ell\alpha^2-q_1^2)] \label{difftheta}
\end{eqnarray}
satisfying homogeneous initial conditions.

Multiply \eq{diffp1} by $\dot p_1$, \eq{diffp2} by $\dot p_2$, \eq{diffy} by $\dot y$, \eq{difftheta} by $\dot\theta$. Then integrate with respect
to $x$ over $(0,\pi)$ and with respect to $t$ over $(0,t)$; we obtain
\begin{eqnarray}
m\|\dot p_1(t)\|_\xi^2\!+\!\|p_1(t)\|_{H^1_\xi}^2 &\!\!=\!\!&
2\frac{AE}{L_c}\int_0^t\left[\int_0^\pi\frac{s'(x)[(q_1^1)'(x,\tau)-(q_1^2)'(x,\tau)]}{\sr}dx\right]
\left[\int_0^\pi\frac{s''(x)}{\sr^3}\dot p_1dx\right]d\tau \notag\\
 & & +2\int_0^t\!\!\int_0^\pi[\Phi(z^1\!+\!\ell\alpha^1\!-\!q_1^1)\!-\!\Phi(z^2\!+\!\ell\alpha^2\!-\!q_1^2)]\dot p_1 dxd\tau \label{boh1} \\
m\|\dot p_2(t)\|_\xi^2\!+\!\|p_2(t)\|_{H^1_\xi}^2 &\!\!=\!\!&
2\frac{AE}{L_c}\int_0^t\left[\int_0^\pi\frac{s'(x)[(q_2^1)'(x,\tau)-(q_2^2)'(x,\tau)]}{\sr}dx\right]
\left[\int_0^\pi\frac{s''(x)}{\sr^3}\dot p_2dx\right]d\tau \notag\\
 & & +2\int_0^t\!\!\int_0^\pi[\Phi(z^1\!-\!\ell\alpha^1\!-\!q_2^1)\!-\!\Phi(z^2\!-\!\ell\alpha^2\!-\!q_2^2)]\dot p_2 dxd\tau \label{boh2} \\
M\|\dot y(t)\|_2^2\!+\!EI\|y(t)\|_{H^2}^2 &\!\!=\!\!& 2\int_0^t\!\!\int_0^\pi[\Phi(z^2\!+\!\ell\alpha^2\!-\!q_1^2)\!
-\!\Phi(z^1\!+\!\ell\alpha^1\!-\!q_1^1)]\dot y dxd\tau \notag \\
\ & & +2\int_0^t\!\!\int_0^\pi[\Phi(z^2\!-\!\ell\alpha^2\!-\!q_2^2)\!-\!\Phi(z^1\!-\!\ell\alpha^1\!-\!q_2^1)]\dot y dxd\tau \label{boh3} \\
\frac{M\ell^2}{3}\|\dot\theta(t)\|_2^2\!+\!GK\|\theta(t)\|_{H^1}^2 &\!\!=\!\!&
2\ell\int_0^t\!\!\int_0^\pi\!\![\Phi(z^1\!-\!\ell\alpha^1\!-\!q_2^1)\!-\!\Phi(z^2\!-\!\ell\alpha^2\!-\!q_2^2)]\dot\theta dxd\tau \notag\\
\ & & -2\ell\int_0^t\!\!\int_0^\pi\!\![\Phi(z^1\!+\!\ell\alpha^1\!-\!q_1^1)\!-\!\Phi(z^2\!+\!\ell\alpha^2\!-\!q_1^2)]
\dot\theta dxd\tau.\label{boh4}
\end{eqnarray}

Our purpose is to add all the above identities and to obtain estimates. By adding the four left hand sides of \eq{boh1}-\eq{boh2}-\eq{boh3}-\eq{boh4},
we find $\gamma>0$ such that
$$m\|\dot p_1(t)\|_\xi^2+\|p_1(t)\|_{H^1_\xi}^2+m\|\dot p_2(t)\|_\xi^2+\|p_2(t)\|_{H^1_\xi}^2+M\, \|\dot y(t)\|_2^2+EI\, \|y(t)\|_{H^2}^2+
\frac{M}{3}\ell^2\|\dot\theta(t)\|_2^2+GK\|\theta(t)\|_{H^1}^2$$
\neweq{lower}
\ge\gamma\Big(\|\dot p_1(t)\|_\xi^2+\|p_1(t)\|_{H^1_\xi}^2+\|\dot p_2(t)\|_\xi^2+\|p_2(t)\|_{H^1_\xi}^2+\|\dot y(t)\|_2^2+\|y(t)\|_{H^2}^2+
\|\dot\theta(t)\|_2^2+\|\theta(t)\|_{H^1}^2\Big)\, .
\endeq

Next, we seek an upper bound for the sum of the four right hand sides. Two kinds of terms appear: the $\Phi$-terms and the nonlocal terms.
For the term containing $\Phi$ in \eq{boh1} we remark that, by the Lagrange Theorem,
$$
|\Phi(z^1+\ell\alpha^1-q_1^1)-\Phi(z^2+\ell\alpha^2-q_1^2)|\le\kappa(x)\Big(|z^1-z^2|+\ell|\alpha^1-\alpha^2|+|q_1^1-q_1^2|\Big)\, ;
$$
therefore, by using the H\"older inequality several times,
$$\left|\int_0^t\!\!\int_0^\pi[\Phi(z^1+\ell\alpha^1-q_1^1)-\Phi(z^2+\ell\alpha^2-q_1^2)]\dot p_1 dxd\tau\right|$$
$$\le c \int_0^t\!\!\int_0^\pi \Big(|z^1-z^2|+\ell|\alpha^1-\alpha^2|+|q_1^1-q_1^2|\Big)|\dot p_1| dxd\tau$$
\neweq{stimap1}
\le c \int_0^t \Big(\|z^1-z^2\|_2+\|\alpha^1-\alpha^2\|_2+\|q_1^1-q_1^2\|_2\Big)\|\dot p_1(\tau)\|_\xi d\tau\, .
\endeq
Similarly, the terms containing $\Phi$ in \eq{boh2}-\eq{boh3}-\eq{boh4} may be estimated, respectively, as follows
\neweq{stimap2}
\mbox{$\Phi$-terms in \eq{boh2} }\le c \int_0^t \Big(\|z^1-z^2\|_2+\|\alpha^1-\alpha^2\|_2+\|q_2^1-q_2^2\|_2\Big)\|\dot p_2(\tau)\|_\xi d\tau\, ,
\endeq
\neweq{stimay}
\mbox{$\Phi$-terms in \eq{boh3} }\le c\int_0^t\Big(\|z^1-z^2\|_2+\|\alpha^1-\alpha^2\|_2+\|q_1^1-q_1^2\|_2+\|q_2^1-q_2^2\|_2\Big)\|\dot y(\tau)\|_2 d\tau\, ,
\endeq
\neweq{stimatheta}
\mbox{$\Phi$-terms in \eq{boh4} }\le c\int_0^t\Big(\|z^1-z^2\|_2+\|\alpha^1-\alpha^2\|_2+\|q_1^1-q_1^2\|_2+\|q_2^1-q_2^2\|_2\Big)\|\dot \theta(\tau)\|_2
d\tau\, .
\endeq

Then, we upper estimate the two nonlocal terms in \eq{boh1}-\eq{boh2}:
$$
\left|\int_0^t\left[\int_0^\pi\frac{s'(x)[(q_i^1)'(x,\tau)-(q_i^2)'(x,\tau)]}{\sr}dx\right]\left[\int_0^\pi\frac{s''(x)}{\sr^3}\dot p_i dx\right]d\tau\right|
$$
$$
\le c\int_0^t\left[\int_0^\pi|(q_i^1)'(x,\tau)-(q_i^2)'(x,\tau)|dx\right]\left[\int_0^\pi|\dot p_i|dx\right]d\tau
$$
\neweq{stimanonlocal}
\le c\int_0^t\|(q_i^1)'(\tau)-(q_i^2)'(\tau)\|_2\|\dot p_i(\tau)\|_\xi d\tau\, .
\endeq

All together, \eq{stimap1}-\eq{stimap2}-\eq{stimay}-\eq{stimatheta}-\eq{stimanonlocal} prove that
$$\mbox{the sum of all the r.h.s.\ of \eq{boh1}-\eq{boh2}-\eq{boh3}-\eq{boh4}}$$
{\scriptsize
$$\le c\int_0^t\!\Big(\|z^1\!-\!z^2\|_2\!+\!\|\alpha^1\!-\!\alpha^2\|_2\!+\!\|q_1^1\!-\!q_1^2\|_2\!+\!\|q_2^1\!-\!q_2^2\|_2\Big)
\Big(\|\dot p_1(\tau)\|_\xi\!+\!\|\dot p_2(\tau)\|_\xi\!+\!\|\dot y(\tau)\|_2\!+\!\|\dot \theta(\tau)\|_2\Big)d\tau$$
$$+c\int_0^t\Big(\|(q_1^1)'(\tau)-(q_1^2)'(\tau)\|_2+\|(q_2^1)'(\tau)-(q_2^2)'(\tau)\|_2\Big)\Big(\|\dot p_1(\tau)\|_\xi+\|\dot p_2(\tau)\|_\xi\Big) d\tau$$
$$\le c\sqrt{t}\Big(\|z^1\!-\!z^2\|_{L^\infty(L^2)}\!+\!\|\alpha^1\!-\!\alpha^2\|_{L^\infty(L^2)}\!+
\!\|q_1^1\!-\!q_1^2\|_{L^\infty(L^2)}\!+\!\|q_2^1\!-\!q_2^2\|_{L^\infty(L^2)}\Big)
\left(\int_0^t\!\big(\|\dot p_1(\tau)\|_\xi^2\!+\!\|\dot p_2(\tau)\|_\xi^2\!+\!\|\dot y(\tau)\|_2^2\!+\!\|\dot\theta(\tau)\|_2^2\big)d\tau\right)^{1/2}$$
$$+c\sqrt{t}\Big(\|(q_1^1)'\!-\!(q_1^2)'\|_{L^\infty(L^2)}\!+\!\|(q_2^1)'\!-\!(q_2^2)'\|_{L^\infty(L^2)}\Big)
\left(\int_0^t\!\Big(\|\dot p_1(\tau)\|_\xi^2\!+\!\|\dot p_2(\tau)\|_\xi^2\Big) d\tau\right)^{1/2}$$
\neweq{upper}
\le c\sqrt{T}\|(q_1^1-q_1^2,q_2^1-q_2^2,z^1-z^2,\alpha^1-\alpha^2)\|_{Z_T}\left(\int_0^t
\big(\|\dot p_1(\tau)\|_\xi^2+\|\dot p_2(\tau)\|_\xi^2+\|\dot y(\tau)\|_2^2+\|\dot\theta(\tau)\|_2^2\big)d\tau\right)^{1/2}\, .
\endeq
}

By taking the sum of \eq{boh1}-\eq{boh2}-\eq{boh3}-\eq{boh4}, by taking into account the lower bound \eq{lower} and the upper bound \eq{upper},
for all $t\in(0,T)$ we obtain
{\scriptsize
$$\|\dot p_1(t)\|_\xi^2+\|p_1(t)\|_{H^1_\xi}^2+\|\dot p_2(t)\|_\xi^2+\|p_2(t)\|_{H^1_\xi}^2+\|\dot y(t)\|_2^2+\|y(t)\|_{H^2}^2+
\|\dot\theta(t)\|_2^2+\|\theta(t)\|_{H^1}^2$$
\neweq{primopasso}
\le c\sqrt{T}\|(q_1^1-q_1^2,q_2^1-q_2^2,z^1-z^2,\alpha^1-\alpha^2)\|_{Z_T}\left(\int_0^t
\big(\|\dot p_1(\tau)\|_\xi^2+\|\dot p_2(\tau)\|_\xi^2+\|\dot y(\tau)\|_2^2+\|\dot\theta(\tau)\|_2^2\big)d\tau\right)^{1/2}\, .
\endeq
}
In particular, by dropping the potential part, \eq{primopasso} implies that
$$\phi'(t)\le K\sqrt{\phi(t)}\quad\mbox{with}\quad
\phi(t)=\int_0^t\big(\|\dot p_1(\tau)\|_\xi^2+\|\dot p_2(\tau)\|_\xi^2+\|\dot y(\tau)\|_2^2+\|\dot\theta(\tau)\|_2^2\big)d\tau$$
$$\mbox{and}\quad K=c\sqrt{T}\|(q_1^1-q_1^2,q_2^1-q_2^2,z^1-z^2,\alpha^1-\alpha^2)\|_{Z_T}\, .$$
This differential inequality yields $\sqrt{\phi(t)}\le Kt/2$ for all $t\in[0,T]$ which, inserted into the right hand side of \eq{primopasso}, gives
$$\|\dot p_1(t)\|_\xi^2+\|p_1(t)\|_{H^1_\xi}^2+\|\dot p_2(t)\|_\xi^2+\|p_2(t)\|_{H^1_\xi}^2+\|\dot y(t)\|_2^2+\|y(t)\|_{H^2}^2+
\|\dot\theta(t)\|_2^2+\|\theta(t)\|_{H^1}^2$$
$$
\le cT^2 \|(q_1^1-q_1^2,q_2^1-q_2^2,z^1-z^2,\alpha^1-\alpha^2)\|_{Z_T}^2\qquad\forall t\in[0,T]\, .
$$
By taking the supremum with respect to $t\in[0,T]$ and by replacing $p_i$, $y$, $\theta$, this finally yields
$$\|(p_1^1-p_1^2,p_2^1-p_2^2,y^1-y^2,\theta^1-\theta^2)\|_{Z_T}\le C\, T\, \|(q_1^1-q_1^2,q_2^1-q_2^2,z^1-z^2,\alpha^1-\alpha^2)\|_{Z_T}\, .$$
Then, for sufficiently small $T$ the map $\Upsilon$ is contractive.
\end{proof}

Since a solution of \eqref{primordine} is a fixed point for $\Upsilon$ (as defined in \eq{Upsilon}), Lemma \ref{contraction} shows that there exists
a unique weak solution $(p_1,p_2,y,\theta)\in X_T$ of \eqref{primordine} satisfying the initial conditions \eqref{initial}, provided that $T>0$ is
sufficiently small. The conservation of the energy in \eq{conserved} shows that the solution cannot blow up in finite time and therefore the solution
is global, that is, $T=\infty$. This completes the proof of Theorem \ref{wp}.

\end{document}